\documentclass[a4paper]{article}
\usepackage{amsmath,amsthm,amssymb}
\usepackage{graphicx}
\usepackage{multirow}
\usepackage{dcolumn}
\newtheorem{Def}{Definition}
\newtheorem{theorem}[Def]{Theorem}
\newtheorem{lemma}[Def]{Lemma}

\newtheorem{remark}[Def]{Remark}
\newtheorem*{maintheorem}{Main Theorem}
\newtheorem*{question}{Question}
\makeatletter
 
\newcolumntype{d}[1]{D{.}{\cdot}{#1}}
\renewcommand{\@cite}[2]{{#1\if@tempswa , #2\fi}}
\makeatother

\title{Generating the mapping class group by torsion elements of small order.}
\author{Naoyuki Monden}
\date{}
\begin{document}
\maketitle
\begin{abstract}
We show that the mapping class group of a closed oriented surface of genus at least three is generated by 3 elements of order 3 and by 4 elements of order 4.
Note that the mapping class group cannot be generated by finitely many torsion elements of same order if genus is equal to one or two.
\end{abstract}

\section{Introduction}
Let $\Sigma_{g}$ denote a closed, oriented surface of genus $g$, and let ${\cal M}_{g}$ denote its mapping class group, which is the group of homotopy classes of orientation-preserving homeomorphisms.

The study of generators of ${\cal M}_{g}$ was pioneered by Dehn.
Dehn [De] proved that ${\cal M}_{g}$ is generated by a finite set of Dehn twists.
Lickorish [Li] showed that $3g-1$ Dehn twists generate ${\cal M}_{g}$.
This number was improved to $2g+1$ by Humphries [Hu].
Humphries proved, moreover, that in fact the number $2g+1$ is minimal;
i.e. ${\cal M}_{g}$ cannot be generated by $2g$ (or less) Dehn twists.

It is classical problem to find small generating sets and torsion generating sets for ${\cal M}_{g}$.
Maclachlan [Ma] proved that the moduli space is simply connected as a topological space by showing that ${\cal M}_{g}$ is generated by torsion elements.
McCarthy and Papadopoulos [MP] proved that ${\cal M}_{g}$ is generated by infinitely many conjugates of a single involution for $g\geq 3$.
Luo [Luo] discovered a first finite set of involutions which generate ${\cal M}_{g}$ for $g\geq 3$.
Luo posed the question of whether there is a universal upper bound, independent of $g$, for the number of torsion elements needs to generate ${\cal M}_{g}$.
Brendle and Farb answered Luo's question.
Brendle and Farb [BF] proved that ${\cal M}_{g}$ is generated by 3 elements of order $2g+2$, $4g+2$ and $2$ (or $g$).
More, Korkmaz [Ko] showed that ${\cal M}_{g}$ is generated by 2 torsion elements, each of order $4g+2$.
Brendle and Farb [BF] also constructed a generating set of ${\cal M}_{g}$ for $g\geq 3$ consisting of 6 involutions.
Kassabov [Ka] improved their method to show that ${\cal M}_{g}$ is generated by 4 involutions if $g\geq 7$, 5 involutions if $g\geq 5$ and 6 involutions if $g\geq 3$.

For all $g\geq 1$, the elements of order 2, 3 and 4 are in ${\cal M}_{g}$.
In the present paper, we construct a generating set of ${\cal M}_{g}$ consisting of elements of small order($>2$).
\begin{maintheorem}
For $g\geq 3$, ${\cal M}_{g}$ can be generated by 3 elements of order 3 and by 4 elements of order 4.
\end{maintheorem}

\section{Preliminaries}
Let $c$ be a simple closed curve on $\Sigma_{g}$. 
Then the (right hand) Dehn twist $T_{c}$ about $c$ is the homotopy class of the homeomorphism obtained
by cutting $\Sigma_{g}$ along $c$, twisting one of the side by $360^{\circ}$ to the right and gluing two sides of $c$ back to each other. Figure~\ref{fig1} shows the Dehn twist about the curve $c$. 
\begin{figure}[htbp]
 \begin{center}
  \includegraphics*[width=8cm]{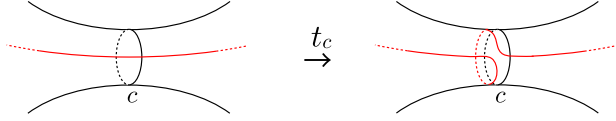}
 \end{center}
 \caption{The Dehn twist}
 \label{fig1}
\end{figure}
We will denote by $T_{c}$ the Dehn twist about the curve $c$. 

We recall the following lemmas and theorems.
\begin{lemma}\label{lem1}
For any homeomorphism $h$ of the surface $\Sigma_{g}$, the Dehn twists around the curves $c$ and $h(c)$ are conjugate in the mapping class group ${\cal M}_{g}$,
\begin{eqnarray*}
T_{h(\textit{c})}=hT_{\textit{c}}h^{-1}.
\end{eqnarray*}
\end{lemma}
\begin{lemma}\label{lem2}
Let $c$ and $d$ be two simple closed curves on $\Sigma_{g}$. If $c$ is disjoint from $d$, then 
\begin{eqnarray*}
T_{c}T_{d}=T_{d}T_{c}
\end{eqnarray*}
\end{lemma}
\begin{lemma}\label{lem3}
If the geometric intersection number of $c$ and $d$ is one, then
\begin{eqnarray*}
T_{c}T_{d}T_{c}=T_{d}T_{c}T_{d}
\end{eqnarray*}
\end{lemma}
\begin{theorem}[Lickorish]
We denote the curves $\alpha_{i}$, $\beta_{i}$, $\gamma_{i}$ as shown in Figure~\ref{fig2}. Then ${\cal M}_{g}$ is generated by $T_{\alpha_{1}},\cdots,T_{\alpha_{g}}, T_{\beta_{1}},\cdots,T_{\beta_{g}}, T_{\gamma_{1}},\cdots,T_{\gamma_{g-1}}$.
\end{theorem}
\begin{figure}[htbp]
 \begin{center}
  \includegraphics*[width=7cm]{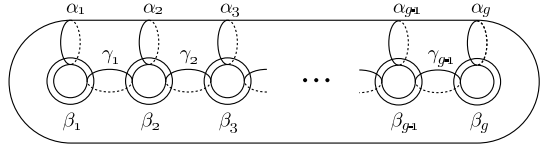}
 \end{center}
 \caption{Lickorish's generators}
 \label{fig2}
\end{figure}
We call $T_{\alpha_{1}},\cdots,T_{\alpha_{g}}, T_{\beta_{1}},\cdots,T_{\beta_{g}}, T_{\gamma_{1}},\cdots,T_{\gamma_{g-1}}$ Lickorish's generators.
\begin{theorem}[Humphries]
${\cal M}_{g}$ is generated by $T_{\alpha_{1}}$, $T_{\alpha_{2}}$, $T_{\beta_{1}},\cdots,T_{\beta_{g}}$, $T_{\gamma_{1}},\cdots,T_{\gamma_{g-1}}$.
\end{theorem}
We call $T_{\alpha_{1}}, T_{\alpha_{2}}, T_{\beta_{1}},\cdots,T_{\beta_{g}}, T_{\gamma_{1}},\cdots,T_{\gamma_{g-1}}$ Humphries's generators.

\section{Generating the mapping class group by 3 elements of order 3}
In this section we prove that the mapping class group ${\cal M}_{g}$ is generated by 3 elements of order 3.
We assume that $g\geq 3$.

\subsection{Construction of elements of order 3}
We construct two elements of order 3 by cutting and gluing surfaces.
We take the curves $x_{1}$, $x_{2}$ and the separating curve $\delta$ like Figure~\ref{fig a}.
\begin{figure}[htbp]
 \begin{center}
  \includegraphics*[width=7cm]{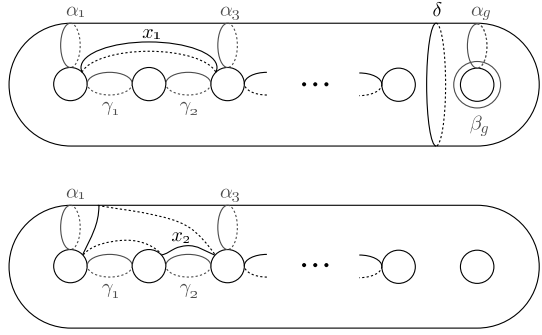}
 \end{center}
 \caption{the curves $x_{1}$, $x_{2}$, $\delta$}
 \label{fig a}
\end{figure}
\subsubsection{Odd genus}
We assume that $g$ is odd.
We construct an element $f$ of order 3.

We cut $\Sigma_{g}$ along the curves $\gamma_{1}, \gamma_{2}, \alpha_{3}$ and $\alpha_{2k}, \gamma_{2k}, \alpha_{2k+1}$ $(k=2,\ldots, \frac{g-1}{2})$ to obtain $\frac{g-1}{2}$ surfaces $S_{1}$, $S_{2}$,\ldots, $S_{\frac{g-1}{2}}$ as shown in Figure~\ref{fig3}.
$S_{1}$ is a sphere with $\frac{3(g+1)}{2}$ boundary components and $S_{k}$ $(k=2,\ldots, \frac{g-1}{2})$ is a pair of pants bounded by $\alpha_{2k}, \gamma_{2k}, \alpha_{2k+1}$.
\begin{figure}[htbp]
 \begin{center}
  \includegraphics*[width=8cm]{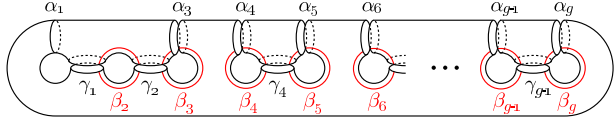}
 \end{center}
 \caption{Cutting the surface}
 \label{fig3}
\end{figure}

We embed $S_{1}$, $S_{2}$,\ldots, $S_{\frac{g-1}{2}}$ in $\mathbb{R}^{3}$ so that they are invariant under $\frac{2\pi}{3}$-rotations $f_{1}$, $f_{2}$,\ldots, $f_{\frac{g-1}{2}}$, respectively (cf. Figure~\ref{fig4}).
\begin{figure}[htbp]
 \begin{center}
  \includegraphics*[width=6cm]{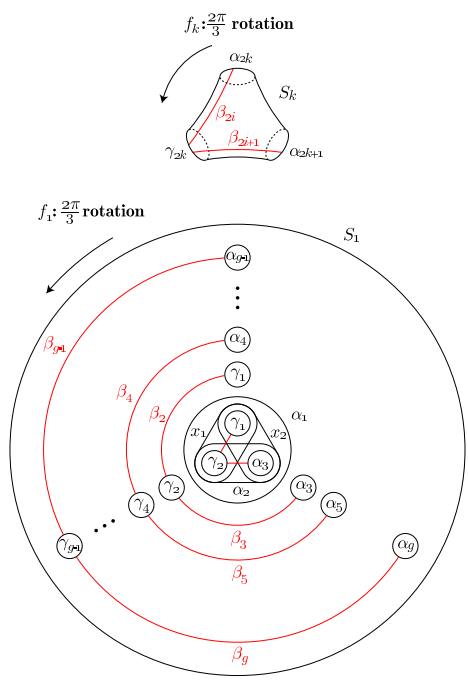}
 \end{center}
 \caption{$\mathbb{Z}_{3}$-symmetry of $\Sigma_{g}$}
 \label{fig4}
\end{figure}

Since the homeomorphisms $f_{1}$, $f_{2}$,\ldots, $f_{\frac{g-1}{2}}$ coincide on the boundaries, they naturally define a homeomorphism $f : \Sigma_{g}\rightarrow \Sigma_{g}$ of order 3.
$f$ acts on the curves on $\Sigma_{g}$ as follows:
\begin{center}
\begin{tabular}{lcr}
$f^{2}(\gamma_{1})=f(\gamma_{2})=\alpha_{3},$& $f(\beta_{2})=\beta_{3},$& $$\\
$f^{2}(\alpha_{2})=f(x_{2})=x_{1},$& $$& $$\\
$f^{2}(\alpha_{2k})=f(\gamma_{2k})=\alpha_{2k+1},$& $f(\beta_{2k})=\beta_{2k+1}$& $(k=2,\ldots, \frac{g-1}{2}).$
\end{tabular}
\end{center}

\

We construct an element $h$ of order 3.

We cut $\Sigma_{g}$ along the curves $\alpha_{2j-1}, \gamma_{2j-1}, \alpha_{2j}$ $(j=1,\ldots, \frac{g-1}{2})$ and $\delta$ to obtain $\frac{g+3}{2}$ surfaces $S_{1}^{\prime}$, $S_{2}^{\prime}$,\ldots, $S_{\frac{g+1}{2}}^{\prime}$, $S_{\frac{g+3}{2}}^{\prime}$ as shown in Figure~\ref{fig5}.
$S_{1}^{\prime}$ is a sphere with $\frac{3g+1}{2}$ boundary components, $S_{j}^{\prime}$ $(j=1,\ldots, \frac{g-1}{2})$ is a pair of pants bounded by $\alpha_{2j-1}, \gamma_{2j-1}, \alpha_{2j}$ and $S_{\frac{g+3}{2}}^{\prime}$ is a torus bounded by $\delta$.
\begin{figure}[htbp]
 \begin{center}
  \includegraphics*[width=8cm]{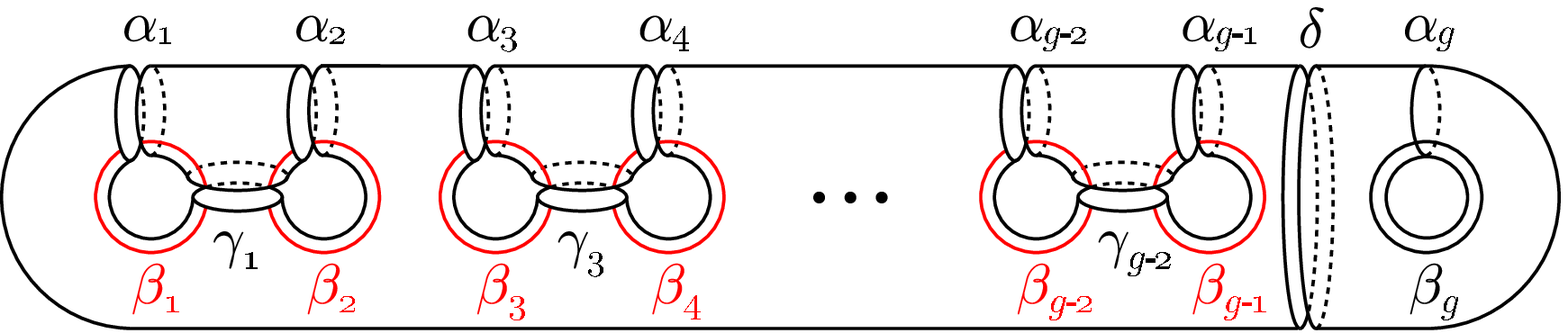}
 \end{center}
 \caption{Cutting the surface, I\hspace{-.1em}I }
 \label{fig5}
\end{figure}

We embed $S_{1}^{\prime}$, $S_{2}^{\prime}$,\ldots, $S_{\frac{g+1}{2}}^{\prime}$ in $\mathbb{R}^{3}$ so that they are invariant under $\frac{2\pi}{3}$-rotations $h_{1}$, $h_{2}$,\ldots, $h_{\frac{g+1}{2}}$, respectively (cf. Figure~\ref{fig6}).
We define that $h_{\frac{g+3}{2}}=(T_{\beta_{g}}T_{\alpha_{g}})^{2}$.
Note that $h_{1}^{3}=T_{\delta}^{-1}$, $h_{\frac{g+3}{2}}^{3}=T_{\delta}$ and $h_{\frac{g+3}{2}}(\alpha_{g})=\beta_{g}$.
\begin{figure}[htbp]
 \begin{center}
  \includegraphics*[width=6cm]{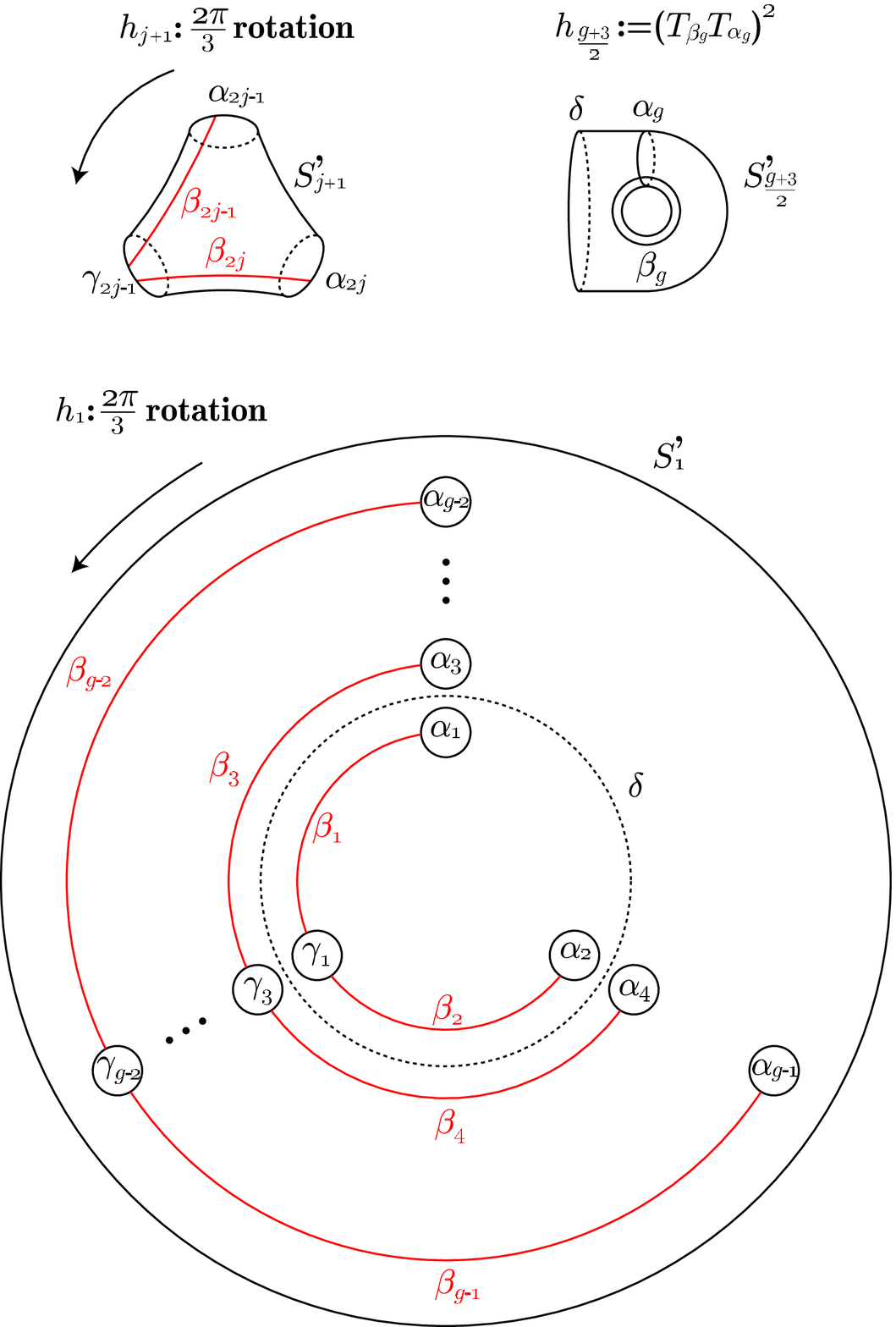}
 \end{center}
 \caption{$\mathbb{Z}_{3}$-symmetry of $\Sigma_{g}$, I\hspace{-.1em}I}
 \label{fig6}
\end{figure}

Since the homeomorphisms $h_{1}$, $h_{2}$,\ldots, $h_{\frac{g+1}{2}}$, $h_{\frac{g+3}{2}}$ coincide on the boundaries, they naturally define a homeomorphism $h:\Sigma_{g}\rightarrow \Sigma_{g}$ of order 3.
$h$ acts on the curves on $\Sigma_{g}$ as follows:
\begin{center}
\begin{tabular}{lcr}
$h(\alpha_{g})=\beta_{g},$& $$& $$\\
$h^{2}(\alpha_{2j-1})=h(\gamma_{2j-1})=\alpha_{2j},$& $h(\beta_{2j-1})=\beta_{2j}$& $(j=1,\ldots, \frac{g-1}{2}).$
\end{tabular}
\end{center}

\subsubsection{Even genus}
We assume that $g$ is even.
By the similar arguments of the case of odd genus we construct $f$ and $h$ which are order 3.

We cut $\Sigma_{g}$ along the curves $\gamma_{1}, \gamma_{2}, \alpha_{3}$, $\delta$ and $\alpha_{2k}, \gamma_{2k}, \alpha_{2k+1}$ $(k=2,\cdots, \frac{g-2}{2})$ to obtain $\frac{g}{2}$ surfaces $S_{1}$, $S_{2}$,\ldots, $S_{\frac{g-2}{2}}$, $S_{\frac{g}{2}}$ as shown in Figure~\ref{fig b}.
$S_{1}$ is a sphere with $\frac{3g+2}{2}$ boundary components, $S_{k}$ $(k=2,\cdots, \frac{g-2}{2})$ is a pair of pants bounded by $\alpha_{2k}, \gamma_{2k}, \alpha_{2k+1}$ and $S_{\frac{g}{2}}$ is a torus bounded by $\delta$.
\begin{figure}[htbp]
 \begin{center}
  \includegraphics*[width=8cm]{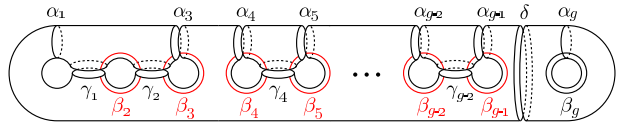}
 \end{center}
 \caption{Cutting the surface, I\hspace{-.1em}I\hspace{-.1em}I}
 \label{fig b}
\end{figure}

We embed $S_{1}$, $S_{2}$,\ldots, $S_{\frac{g-2}{2}}$ in $\mathbb{R}^{3}$ so that they are invariant under $\frac{2\pi}{3}$-rotations $f_{1}$, $f_{2}$,\ldots, $f_{\frac{g-2}{2}}$, respectively (cf. Figure~\ref{fig7}).
We define that $f_{\frac{g}{2}}=(T_{\beta_{g}}T_{\alpha_{g}})^{2}$.
\begin{figure}[htbp]
 \begin{center}
  \includegraphics*[width=6cm]{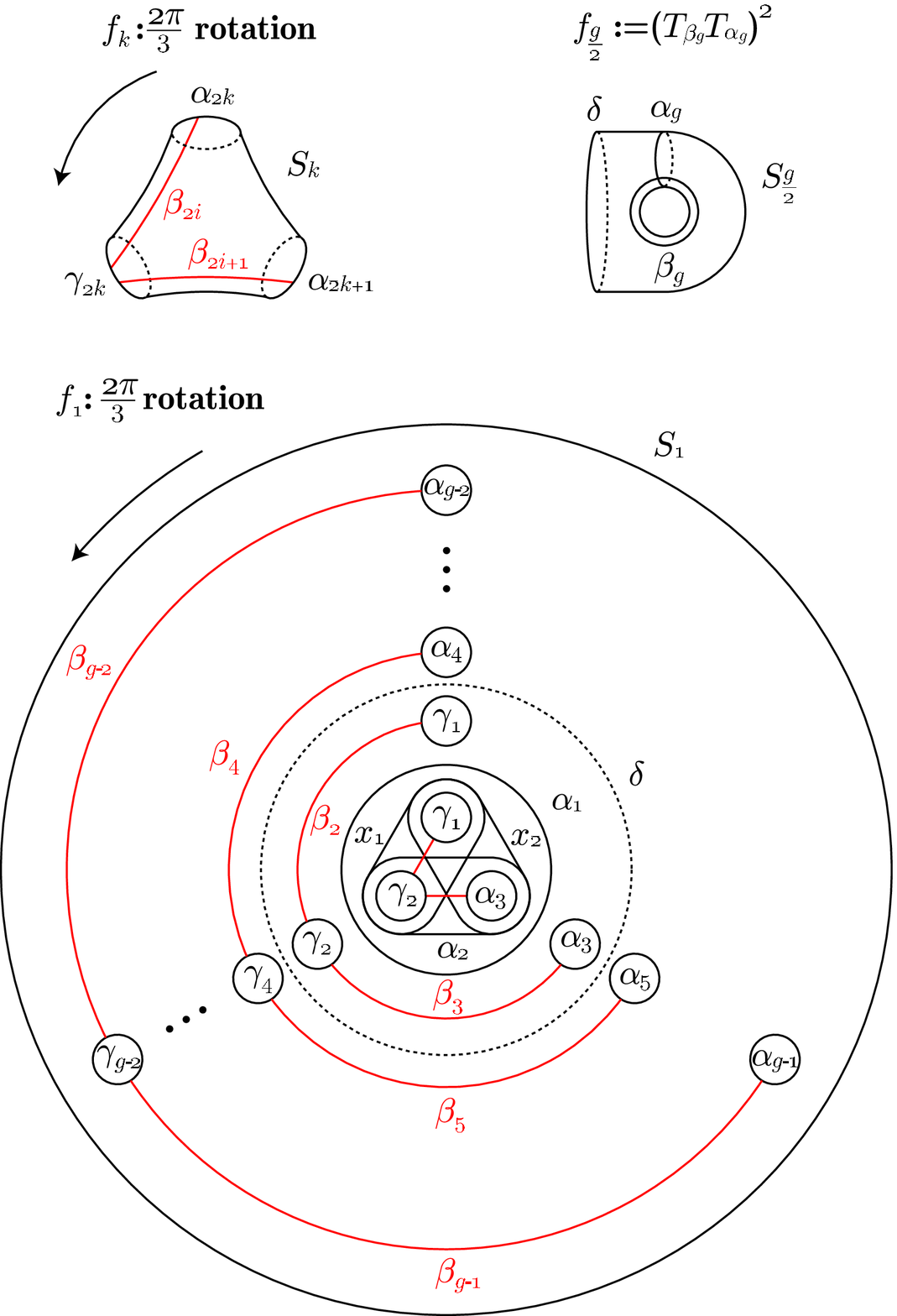}
 \end{center}
 \caption{$\mathbb{Z}_{3}$-symmetry of $\Sigma_{g}$, I\hspace{-.1em}I\hspace{-.1em}I}
 \label{fig7}
\end{figure}

Since the homeomorphisms $f_{1}$, $f_{2}$,\ldots, $f_{\frac{g-2}{2}}$, $f_{\frac{g}{2}}$ coincide on the boundaries, they naturally define a homeomorphism $f:\Sigma_{g}\rightarrow \Sigma_{g}$ of order 3.
$f$ acts on the curves on $\Sigma_{g}$ as follows:
\begin{center}
\begin{tabular}{lcr}
$f^{2}(\gamma_{1})=f(\gamma_{2})=\alpha_{3},$& $f(\beta_{2})=\beta_{3},$& $$\\
$f^{2}(\alpha_{2})=f(x_{2})=x_{1},$& $f(\alpha_{g})=\beta_{g},$& $$\\
$f^{2}(\alpha_{2k})=f(\gamma_{2k})=\alpha_{2k+1},$& $f(\beta_{2k})=\beta_{2k+1}$& $(k=2,\cdots, \frac{g-2}{2}).$
\end{tabular}
\end{center}

\

We cut $\Sigma_{g}$ along the curves $\alpha_{2j-1}, \gamma_{2j-1}, \alpha_{2j}$ $(j=1,\cdots, \frac{g}{2})$ to obtain $\frac{g+2}{2}$ surfaces $S_{1}^{\prime}$, $S_{2}^{\prime}$,\ldots, $S_{\frac{g+2}{2}}^{\prime}$ as shown in Figure~\ref{fig c}.
$S_{1}^{\prime}$ is a sphere with $\frac{3g}{2}$ boundary components and $S_{j+1}^{\prime}$ $(j=1,\cdots, \frac{g}{2})$ is a pair of pants bounded by $\alpha_{2j-1}, \gamma_{2j-1}, \alpha_{2j}$.
\begin{figure}[htbp]
 \begin{center}
  \includegraphics*[width=7cm]{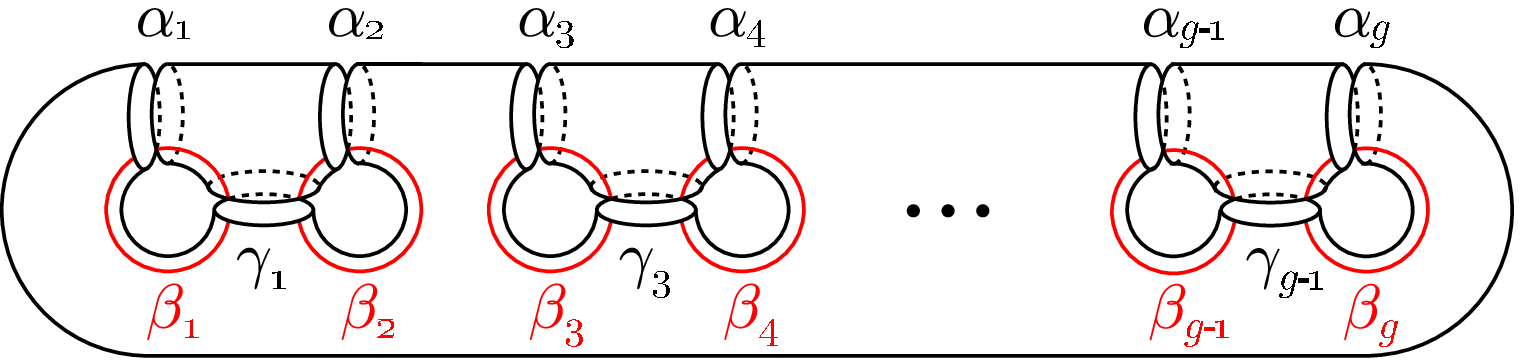}
 \end{center}
 \caption{Cutting the surface, I\hspace{-.1em}V}
 \label{fig c}
\end{figure}

We embed $S_{1}^{\prime}$, $S_{2}^{\prime}$,\ldots, $S_{\frac{g+2}{2}}^{\prime}$ in $\mathbb{R}^{3}$ so that they are invariant under $\frac{2\pi}{3}$-rotations $h_{1}$, $h_{2}$,\ldots, $h_{\frac{g+2}{2}}$, respectively (cf. Figure~\ref{fig8}).
\begin{figure}[htbp]
 \begin{center}
  \includegraphics*[width=6cm]{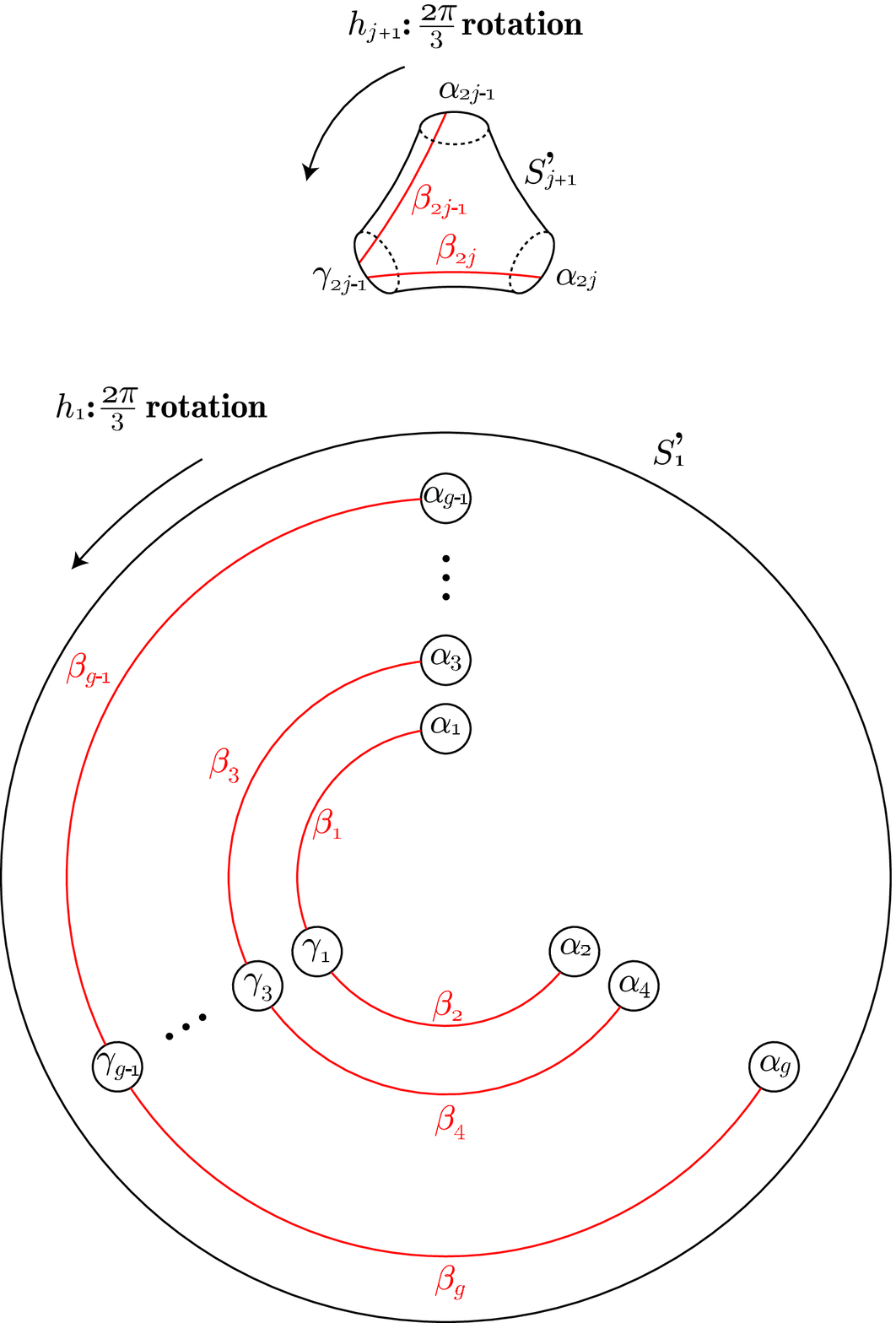}
 \end{center}
 \caption{$\mathbb{Z}_{3}$-symmetry of $\Sigma_{g}$, I\hspace{-.1em}V}
 \label{fig8}
\end{figure}

Since the homeomorphisms $h_{1}$,$h_{2}$,\ldots, $h_{\frac{g+2}{2}}$ coincide on the boundaries, they naturally define a homeomorphism $h:\Sigma_{g}\rightarrow \Sigma_{g}$ of order 3.
$h$ acts on the curves on $\Sigma_{g}$ as follows:
\begin{center}
\begin{tabular}{lcr}
$h^{2}(\alpha_{2j-1})=h(\gamma_{2j-1})=\alpha_{2j},$& $h(\beta_{2j-1})=\beta_{2j}$& $(j=1,\cdots, \frac{g}{2}).$
\end{tabular}
\end{center}

\subsection{Generating the Dehn twist by 3 elemensts of order 3}
We generate the Dehn twist by 3 elements of order 3.
The basic idea is to use the \textit{lantern relation} which was discovered by Dehn and rediscovered by Johnson [Jo].
\begin{figure}[htbp]
 \begin{center}
  \includegraphics*[width=2cm]{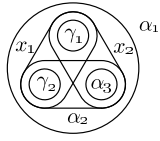}
 \end{center}
 \caption{The Lantern Relation}
 \label{fig9}
\end{figure}

The \textit{lantern relation} is read as follows :
\begin{eqnarray*}
T_{\alpha_{1}}T_{\gamma_{1}}T_{\gamma_{2}}T_{\alpha_{3}}=T_{\alpha_{2}}T_{x_{1}}T_{x_{2}}.
\end{eqnarray*}
where the curves $\alpha_{1}$, $\alpha_{2}$, $\alpha_{3}$, $\gamma_{1}$, $\gamma_{2}$, $x_{1}$ and $x_{2}$ are shown in Figure~\ref{fig a} and Figure~\ref{fig9}.
Since $\alpha_{1},\gamma_{1},\gamma_{2}$ and $\alpha_{3}$ are disjoint each other and $\alpha_{2}, x_{1}$ and $x_{2}$, by Lemma~\ref{lem2} we can rewrite the relation as
\begin{eqnarray}\label{eq1}
T_{\alpha_{1}}=(T_{\alpha_{2}}T_{\gamma_{1}}^{-1})(T_{x_{1}}T_{\alpha_{3}}^{-1})(T_{x_{2}}T_{\gamma_{2}}^{-1}).
\end{eqnarray}
We can find that $f^{2}(\alpha_{2})=x_{1}$, $f^{2}(\gamma_{1})=\alpha_{3}$, $f(\alpha_{2})=x_{2}$ and $f(\gamma_{1})=\gamma_{2}$ from the argument of Section 3.1.
By using Lemma~\ref{lem1} we see that
\begin{eqnarray*}
(T_{x_{1}}T_{\alpha_{3}}^{-1})=f^{2}(T_{\alpha_{2}}T_{\gamma_{1}}^{-1})f^{-2}\\ 
(T_{x_{2}}T_{\gamma_{2}}^{-1})=f(T_{\alpha_{2}}T_{\gamma_{1}}^{-1})f^{-1}. 
\end{eqnarray*}
Since $h$ maps $\gamma_{1}$ to $\alpha_{2}$, we see that
\begin{eqnarray*}
T_{\alpha_{2}}=hT_{\gamma_{1}}h^{-1}
\end{eqnarray*}
and
\begin{eqnarray*}
T_{\alpha_{2}}T_{\gamma_{1}}^{-1}=hT_{\gamma_{1}}h^{-1}T_{\gamma_{1}}^{-1}=h(T_{\gamma_{1}}h^{-1}T_{\gamma_{1}}^{-1}).
\end{eqnarray*}

Let $\bar{h}$ denote $T_{\gamma_{1}}h^{-1}T_{\gamma_{1}}^{-1}$.
We can now rewrite (\ref{eq1}) as
\begin{eqnarray}\label{eq4}
T_{\alpha_{1}}=(h\bar{h})(f^{2}h\bar{h}f^{-2})(fh\bar{h}f^{-1}).
\end{eqnarray}
and hence $T_{\alpha_{1}}$ is a product of 3 elements of order 3.

\subsection{Proof that 3 elements of order 3 generate}
We prove Theorem.
\begin{theorem}
If $g\geq 3$, ${\cal M}_{g}$ is generated by $f$, $h$ and $\bar{h}$.
\end{theorem}
\begin{proof}
Let $G_{1}$ denote the group generated by $f$, $h$ and $\bar{h}$.
By the relation~(\ref{eq4}), $T_{\alpha_{1}}$ is in $G_{1}$.

Let $\alpha$ and $\beta$ be simple closed curves on $\Sigma_{g}$.
The symbol $\alpha \overset{f}{\longrightarrow} \beta$ (resp. $\alpha \overset{h}{\longrightarrow} \beta$) means that $f(\alpha)=\beta$ (resp. $h(\alpha)=\beta$).

The Figure~\ref{fig10} shows that we can send $\alpha_{1}$ to all $\alpha_{i}$ and $\gamma_{i}$ by $f$ and $h$.
\begin{figure}[htbp]
 \begin{center}
  \includegraphics*[width=10cm]{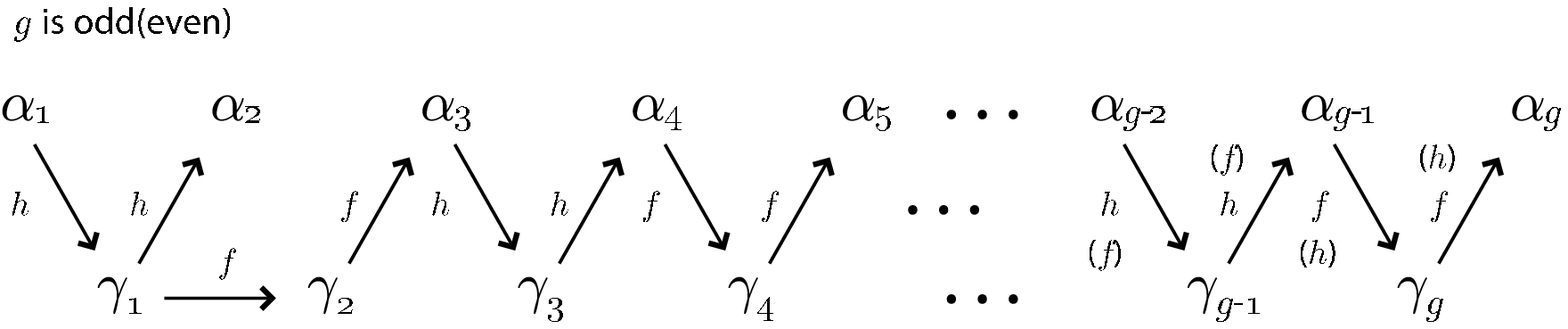}
 \end{center}
 \caption{}
 \label{fig10}
\end{figure}
Therefore, $T_{\alpha_{i}}, T_{\gamma_{i}}\in G_{1}$ for all $i$.

Since $h$ (resp. $f$) maps $\alpha_{g}$ to $\beta_{g}$ in the case of odd (resp. even) genus, $T_{\beta_{g}}\in G_{1}$.
As shown on Figure~\ref{fig11}, we can find that $\beta_{g}$ can be send to $\beta_{i}$ for all $i$ by $f$ and $h$. 
\begin{figure}[htbp]
 \begin{center}
  \includegraphics*[width=10cm]{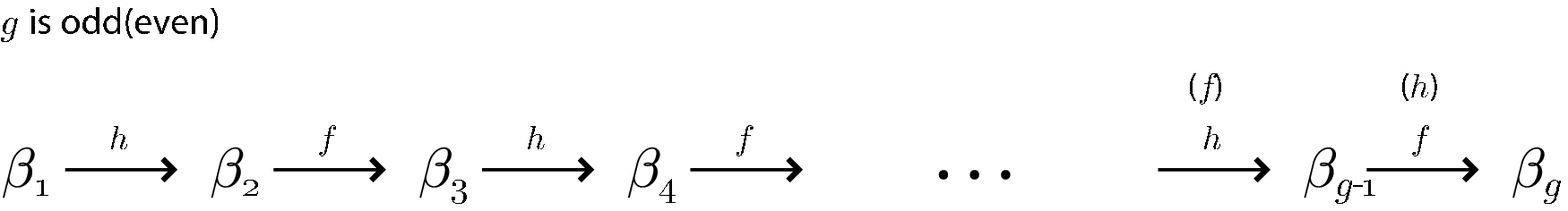}
 \end{center}
 \caption{}
 \label{fig11}
\end{figure}
These shows that $T_{\beta_{i}}\in G_{1}$ for all $i$.

Since we show that all Lickorish's generators are in $G_{1}$, $G_{1}$ is equal to ${\cal M}_{g}$.
\end{proof}

\section{Generating the mapping class group by 4 elements by order 4}
In this section we prove that ${\cal M}_{g}$ can be generated by 4 elements of order 4.
The key point is to use \textit{chain relation}.
\subsection{Construction of elements of order 4}
We prepare to construct 3 elements of order 4.

We recall \textit{chain relation}.
We say that an ordered set of $c_{1}, \ldots, c_{n}$ of simple closed curves on $\Sigma_{g}$ forms an \textit{n-chain} if the geometric intersection $(c_{k}, c_{k+1})=1$ for $k=1, \ldots, n-1$ and $(c_{k}, c_{l})=0$ if $\vert k-l \vert \geq 2$.
If $n$ is odd, the boundary of a regular neighborhood of any $n$-chain has two componets $d_{1}$ and $d_{2}$.

The \textit{chain relation} is read as follow :

For a given $n$-chain $c_{1}, \ldots, c_{n}$, if $n$ is odd we have
\begin{eqnarray*}
(T_{c_{1}}T_{c_{2}}\cdots T_{c_{n}})^{n+1}=T_{d_{1}}T_{d_{2}}
\end{eqnarray*}

We denote the simple closed curves $\alpha^{\prime}_{1},\ldots ,\alpha^{\prime}_{g}$, $s$ and the separating curves $\delta_{1},\ldots ,\delta_{g-1}$ as shown on Figure~\ref{fig13}.
\begin{figure}[htbp]
 \begin{center}
  \includegraphics*[width=6cm]{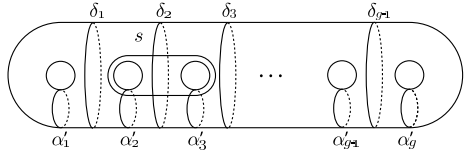}
 \end{center}
 \caption{The simple closed curves $\alpha^{\prime}_{1},\ldots ,\alpha^{\prime}_{g}$, $s$ and the separating curves $\delta_{1},\ldots ,\delta_{g-1}$}
 \label{fig13}
\end{figure}

We note that $\alpha_{1}=\alpha^{\prime}_{1}$ and $\alpha_{g}=\alpha^{\prime}_{g}$.

\

For $i=1, 2,\ldots ,g-2$ we define
\begin{eqnarray*}
\rho_{i}=(T_{\alpha_{i+2}}T_{\beta_{i+2}}T_{\gamma_{i+1}}T_{\beta_{i+1}}T_{\gamma_{i}}T_{\beta_{i}}T_{\alpha^{\prime}_{i}})^{2}.
\end{eqnarray*}
We can find that the boundary components of a regular neighborhood of 7-chain 
$\alpha_{i+2}$, $\beta_{i+2}$, $\gamma_{i+1}$, $\beta_{i+1}$, $\gamma_{i}$, $\beta_{i}$, $\alpha^{\prime}_{i}$ are $\delta_{i+2}$ and $\delta_{i-1}$ like Figure~\ref{fig14}.
\begin{figure}[htbp]
 \begin{center}
  \includegraphics*[width=4cm]{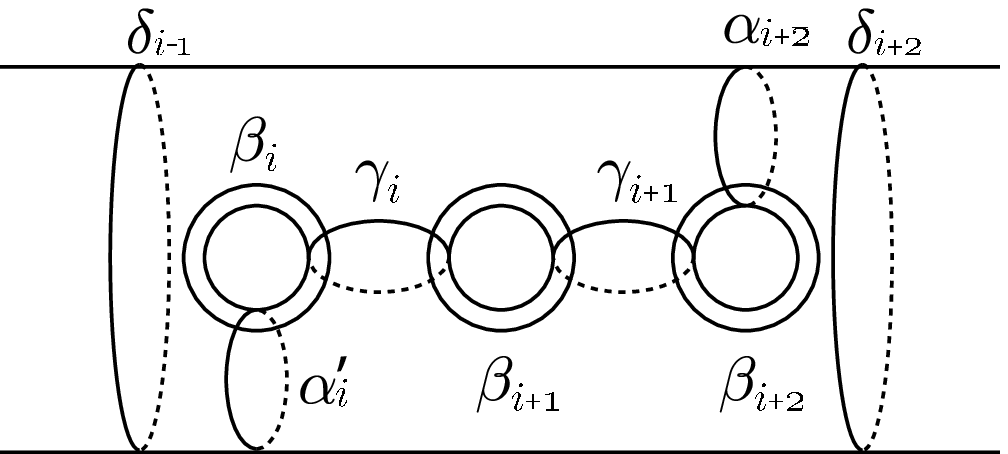}
 \end{center}
 \caption{}
 \label{fig14}
\end{figure}
We note that $\delta_{0}$ and $\delta_{g}$ is trivial.
By \textit{chain relation}, we have
\begin{eqnarray*}
\rho_{i}^{4}=(T_{\alpha_{i+2}}T_{\beta_{i+2}}T_{\gamma_{i+1}}T_{\beta_{i+1}}T_{\gamma_{i}}T_{\beta_{i}}T_{\alpha^{\prime}_{i}})^{8}=T_{\delta_{i+2}}T_{\delta_{i-1}}.
\end{eqnarray*}

For each $i$ we see that $\rho_{i}$ acts on the curves on $\Sigma_{g}$ as follow:
\begin{eqnarray*}
\rho_{i}^{3}(\alpha_{i+2})=\rho_{i}^{2}(\gamma_{i+1})=\rho_{i}(\gamma_{i})=\alpha^{\prime}_{i}, && \rho_{i}^{2}(\beta_{i+2})=\rho_{i}(\beta_{i+1})=\beta_{i}.
\end{eqnarray*}
In paticular we note that $\rho_{1}(\alpha_{2})=x_{1}$.

\

We define 
\begin{eqnarray*}
\rho^{\prime}=(T_{\gamma_{2}}T_{\beta_{3}}T_{\alpha_{3}}T_{s}T_{\gamma_{1}}T_{\beta_{1}}T_{\alpha^{\prime}_{1}})^{2}.
\end{eqnarray*}
We see that the boundary component of a regular neighborhood of 7-chain $\gamma_{2}$, $\beta_{3}$, $\alpha_{3}$, $s$, $\gamma_{2}$, $\beta_{2}$, $\alpha^{\prime}_{1}$ in Figure~\ref{fig15} is $\delta_{3}$.
By \textit{chain relation} we have $(\rho^{\prime})^{4}=T_{\delta_{3}}$.
\begin{figure}[htbp]
 \begin{center}
  \includegraphics*[width=4cm]{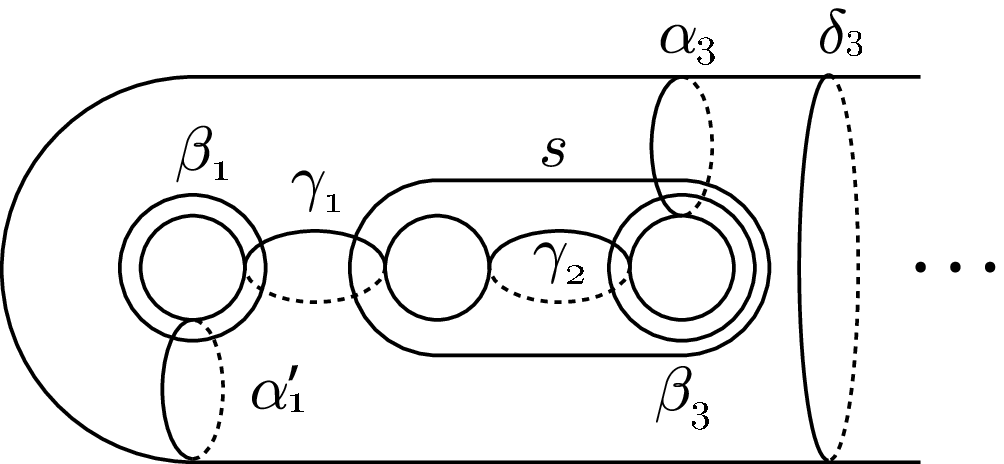}
 \end{center}
 \caption{}
 \label{fig15}
\end{figure}

$\rho^{\prime}$ acts on the curves on $\Sigma_{g}$ as follow:
\begin{eqnarray*}
(\rho^{\prime})^{3}(\gamma_{2})=(\rho^{\prime})^{2}(\alpha_{3})=\rho^{\prime}(\gamma_{1})=\alpha^{\prime}_{1}, && (\rho^{\prime})^{3}(x_{2})=(\rho^{\prime})^{2}(\alpha_{2})=\rho^{\prime}(x_{2})=\alpha_{2}.
\end{eqnarray*}
In paticular we note that $\rho^{\prime}(\alpha_{2})=x_{2}$.

\

We define
\begin{eqnarray*}
\tau_{i}&=&(T_{\alpha^{\prime}_{i+2}}T_{\beta_{i+2}}T_{\gamma_{i+1}}T_{\beta_{i+1}}T_{\gamma_{i}}T_{\beta_{i}}T_{\alpha_{i}})^{2}\\
\end{eqnarray*}
We can find that the boundary components of a regular neighborhood of 7-chain 
$\alpha^{\prime}_{i+2}$, $\beta_{i+2}$, $\gamma_{i+1}$, $\beta_{i+1}$, $\gamma_{i}$, $\beta_{i}$, $\alpha_{i}$ are $\delta_{i+2}$ and $\delta_{i-1}$ like Figure~\ref{fig16}.
By \textit{chain relation} we have $\tau_{i}^{4}$$=$$(T_{\alpha^{\prime}_{i+2}}T_{\beta_{i+2}}T_{\gamma_{i+1}}T_{\beta_{i+1}}T_{\gamma_{i}}T_{\beta_{i}}T_{\alpha_{i}})^{8}$$=$$T_{\delta_{i+2}}T_{\delta_{i-1}}$.
\begin{figure}[htbp]
 \begin{center}
  \includegraphics*[width=4cm]{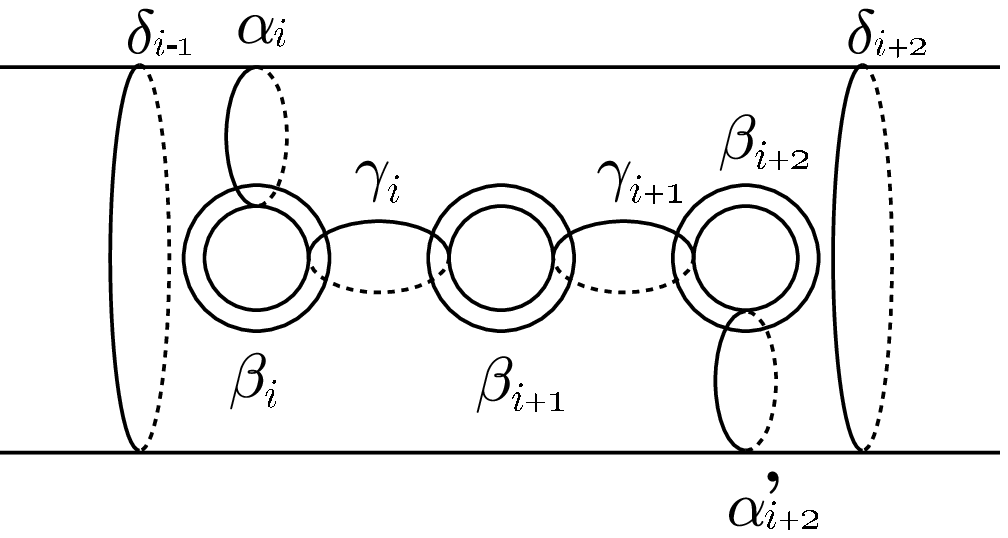}
 \end{center}
 \caption{}
 \label{fig16}
\end{figure}

For each $i$ we see that $\tau_{i}$ acts on the curves on $\Sigma_{g}$ as follow:
\begin{eqnarray*}
\tau^{3}(\alpha^{\prime}_{i+2})=\tau^{2}(\gamma_{i+1})=\tau(\gamma_{i})=\alpha_{i} && \tau^{2}(\beta_{i+2})=\tau(\beta_{i+1})=\beta_{i}.
\end{eqnarray*}

\

For $j=1, 2,\ldots, g$ we define 
\begin{eqnarray*}
\sigma_{j}&=&T_{\alpha_{j}}T_{\beta_{j}}T_{\alpha^{\prime}_{j}}.
\end{eqnarray*}
We can find that the boundary components of a regular neighborhood of 3-chain $\alpha_{j}$, $\beta_{j}$, $\alpha^{\prime}_{j}$ like Figure~\ref{fig17} is $\delta_{j}$ and $\delta_{j-1}$.
By Lemma~\ref{lem3} and \textit{chain relation} we have $\tau^{4}_{j}=T_{\delta_{j}}T_{\delta_{j-1}}$.
\begin{figure}[htbp]
 \begin{center}
  \includegraphics*[width=2cm]{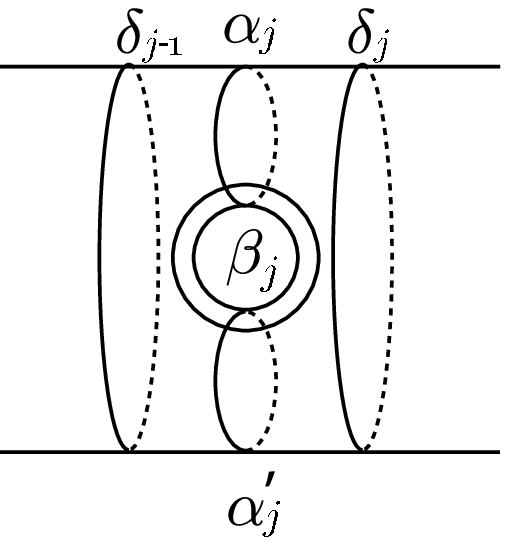}
 \end{center}
 \caption{}
 \label{fig17}
\end{figure}

For each $j$ we see that $\sigma_{j}(\alpha_{j})$$=$$T_{\alpha_{j}}T_{\beta_{j}}T_{\alpha^{\prime}_{j}}(\alpha_{j})$$=$$\beta_{j}$.

\subsubsection{The genus is $3m$}
We assume that $g=3m$.
We construct an element $\phi$ of order 4.
We define
\begin{equation*}
\phi=
\begin{cases}
\rho_{g-2}^{-1}\rho_{g-5}\cdots \rho_{7}^{-1}\rho_{4}\rho_{1}^{-1} & \text{($m (=\frac{g}{3})$ is odd)} \\
\rho_{g-2}\rho_{g-5}^{-1}\cdots \rho_{7}^{-1}\rho_{4}\rho_{1}^{-1} & \text{($m$ is even)}. 
\end{cases}
\end{equation*}

We can find that $\phi^{4}=T_{\delta_{g-3}}^{\mp 1}(T_{\delta_{g-3}}T_{\delta_{g-6}})^{\pm 1}\cdots (T_{\delta_{9}}T_{\delta_{6}})^{-1}(T_{\delta_{6}}T_{\delta_{3}})T_{\delta_{3}}^{-1}=1$.

$\phi$ acts on the curves on $\Sigma_{g}$ as follow:
\begin{center}
\begin{tabular}{lcr}
$\alpha_{2}=\phi(x_{1}),$& $$& $$\\
$\phi^{3}(\alpha^{\prime}_{1})=\phi^{2}(\gamma_{1})=\phi(\gamma_{2})=\alpha_{3},$& $\phi^{2}(\beta_{1})=\phi(\beta_{2})=\beta_{3}$& $$\\
$\alpha^{\prime}_{4}=\phi(\gamma_{4})=\phi^{2}(\gamma_{5})=\phi^{3}(\alpha_{6}),$& $\beta_{4}=\phi(\beta_{5})=\phi^{2}(\beta_{6})$& $$\\
$\vdots$&$\vdots$\\
$\phi^{3}(\alpha^{\prime}_{g-2})=\phi^{2}(\gamma_{g-2})=\phi(\gamma_{g-1})=\alpha_{g},$& $\phi^{2}(\beta_{g-2})=\phi(\beta_{g-1})=\beta_{g}$& $(m \ is \ odd),$\\
\\
$(\alpha^{\prime}_{g-2}=\phi(\gamma_{g-2})=\phi^{2}(\gamma_{g-1})=\phi^{3}(\alpha_{g}),$& $\beta_{g-2}=\phi(\beta_{g-1})=\phi^{2}(\beta_{g})$& $(m \ is \ even))$.
\end{tabular}
\end{center}

\

We construct an element $\psi$ of order 4.
We define
\begin{equation*}
\psi=
\begin{cases}
\rho_{g-2}^{-1}\rho_{g-5}\cdots \rho_{7}^{-1}\rho_{4}\rho^{\prime -1} & \text{($m (=\frac{g}{3})$ is odd)} \\
\rho_{g-2}\rho_{g-5}^{-1}\cdots \rho_{7}^{-1}\rho_{4}\rho^{\prime -1} & \text{($m$ is even)}. 
\end{cases}
\end{equation*}

We can find that $\psi^{4}=1$.
Since $\rho^{\prime}(\alpha_{2})=x_{2}$ and $\rho^{\prime}(\gamma_{2})=\alpha_{3}$, we can easily find that $\psi^{-1}(\alpha_{2})=x_{2}$ and $\psi^{-1}(\gamma_{2})=\alpha_{3}$.

\

We construct an element $\omega$ of order 4.
We define
\begin{equation*}
\omega=
\begin{cases}
\sigma_{g}^{-1}\sigma_{g-1}\tau_{g-5}^{-1}\tau_{g-8}\cdots \tau_{8}\tau_{5}^{-1}\tau_{2}\sigma_{1}^{-1} & \text{($m (=\frac{g}{3})$ is odd)} \\
\sigma_{g}\sigma_{g-1}^{-1}\tau_{g-5}\tau_{g-8}^{-1}\cdots \tau_{8}\tau_{5}^{-1}\tau_{2}\sigma_{1}^{-1} & \text{($m$ is even)}. 
\end{cases}
\end{equation*}
We can find that $\omega^{4}=1$.

$\omega$ acts on the curves on $\Sigma_{g}$ as follow:
\begin{center}
\begin{tabular}{lcr}
$\alpha_{1}=\omega(\beta_{1}),$\\
$\omega^{3}(\alpha_{2})=\omega^{2}(\gamma_{2})=\omega(\gamma_{3})=\alpha^{\prime}_{4},$& $\omega^{2}(\beta_{2})=\omega(\beta_{3})=\beta_{4},$\\
$\alpha_{5}=\omega(\gamma_{5})=\omega^{2}(\gamma_{6})=\omega^{3}(\alpha^{\prime}_{7}),$& $\beta_{5}=\omega(\beta_{6})=\omega^{2}(\beta_{7}),$\\
$\ \ \ \ \ \ \vdots$&$\vdots$\\
$\omega^{3}(\alpha_{g-4})=\omega^{2}(\gamma_{g-4})=\omega(\gamma_{g-3})=\alpha^{\prime}_{g-2},$& $\omega^{2}(\beta_{g-4})=\omega(\beta_{g-3})=\beta_{g-2}$& $(m \ is \ odd),$\\
\\
$(\alpha_{g-4}=\omega(\gamma_{g-4})=\omega^{2}(\gamma_{g-3})=\omega^{3}(\alpha^{\prime}_{g-2}),$& $\beta_{g-4}=\omega(\beta_{g-3})=\omega^{2}(\beta_{g-2})$& $(m \ is \ even)).$
\end{tabular}
\end{center}

\subsubsection{The genus is $3m+1$}
We assume that $g=3m+1$.

We define
\begin{equation*}
\phi=
\begin{cases}
\sigma_{g}\rho_{g-3}^{-1}\rho_{g-6}\cdots \rho_{7}^{-1}\rho_{4}\rho_{1}^{-1} & \text{($m (=\frac{g-1}{3})$ is odd)} \\
\sigma_{g}^{-1}\rho_{g-3}\rho_{g-6}^{-1}\cdots \rho_{7}^{-1}\rho_{4}\rho_{1}^{-1} & \text{($m$ is even)}. 
\end{cases}
\end{equation*}

$\phi$ acts on the curves on $\Sigma_{g}$ as follow:
\begin{center}
\begin{tabular}{lcr}
$\alpha_{2}=\phi(x_{1}),$\\
$\phi^{3}(\alpha^{\prime}_{1})=\phi^{2}(\gamma_{1})=\phi(\gamma_{2})=\alpha_{3},$& $\phi^{2}(\beta_{1})=\phi(\beta_{2})=\beta_{3},$\\
$\alpha^{\prime}_{4}=\phi(\gamma_{4})=\phi^{2}(\gamma_{5})=\alpha_{6},$& $\beta_{4}=\phi(\beta_{5})=\phi^{2}(\beta_{6})$\\
$\ \ \ \ \ \ \vdots$&$\vdots$\\
$\phi^{3}(\alpha^{\prime}_{g-3})=\phi^{2}(\gamma_{g-3})=\phi(\gamma_{g-2})=\alpha_{g-1},$& $\phi^{2}(\beta_{g-3})=\phi(\beta_{g-2})=\beta_{g-1}$& $(m \ is \ odd),$\\
\\
$(\alpha^{\prime}_{g-3}=\phi(\gamma_{g-3})=\phi^{2}(\gamma_{g-2})=\phi^{3}(\alpha_{g-1}),$& $\beta_{g-3}=\phi(\beta_{g-2})=\phi^{2}(\beta_{g-1})$& $(m \ is \ even)).$
\end{tabular}
\end{center}

\

We define
\begin{equation*}
\psi=
\begin{cases}
\rho_{g-2}^{-1}\rho_{g-5}\cdots \rho_{7}^{-1}\rho_{4}\rho^{\prime -1} & \text{($m (=\frac{g-1}{3})$ is odd)} \\
\rho_{g-2}\rho_{g-5}^{-1}\cdots \rho_{7}^{-1}\rho_{4}\rho^{\prime -1} & \text{($m$ is even)}. 
\end{cases}
\end{equation*}

We note that $\psi^{-1}(\alpha_{2})=x_{2}$ and $\psi^{-1}(\gamma_{2})=\alpha_{3}$.

\

We define
\begin{equation*}
\omega=
\begin{cases}
\tau_{g-2}^{-1}\tau_{g-5}\cdots \tau_{8}\tau_{5}^{-1}\tau_{2}\sigma_{1}^{-1} & \text{($m (=\frac{g-1}{3})$ is odd)} \\
\tau_{g-2}\tau_{g-5}^{-1}\cdots \tau_{8}\tau_{5}^{-1}\tau_{2}\sigma_{1}^{-1} & \text{($m$ is even)}. 
\end{cases}
\end{equation*}

$\omega$ acts on the curves on $\Sigma_{g}$ as follow:
\begin{center}
\begin{tabular}{lcr}
$\alpha_{1}=\omega(\beta_{1}),$\\
$\omega^{3}(\alpha_{2})=\omega^{2}(\gamma_{2})=\omega(\gamma_{3})=\alpha^{\prime}_{4},$& $\omega^{2}(\beta_{2})=\omega(\beta_{3})=\beta_{4},$\\
$\alpha_{5}=\omega(\gamma_{5})=\omega^{2}(\gamma_{6})=\omega^{3}(\alpha^{\prime}_{7}),$& $\beta_{5}=\omega(\beta_{6})=\omega^{2}(\beta_{7}),$\\
$\ \ \ \ \ \ \vdots$&$\vdots$\\
$\omega^{3}(\alpha_{g-2})=\omega^{2}(\gamma_{g-2})=\omega(\gamma_{g-1})=\alpha^{\prime}_{g},$& $\omega^{2}(\beta_{g-2})=\omega(\beta_{g-1})=\beta_{g}$& $(m \ is \ odd),$\\
\\
$(\alpha_{g-2}=\omega(\gamma_{g-2})=\omega^{2}(\gamma_{g-1})=\alpha^{\prime}_{g},$& $\omega^{2}(\beta_{g-2})=\omega(\beta_{g-1})=\beta_{g}$& $(m \ is \ even)).$
\end{tabular}
\end{center}

\subsubsection{The genus is $3m+2$}
We assume that $g=3m+2$ $(m\geq 2)$.

We define
\begin{equation*}
\phi=
\begin{cases}
\sigma_{g}^{-1}\sigma_{g-1}\rho_{g-4}^{-1}\rho_{g-7}\cdots \rho_{7}^{-1}\rho_{4}\rho_{1}^{-1} & \text{($m (=\frac{g-2}{3})$ is odd)} \\
\sigma_{g}\sigma_{g-1}^{-1}\rho_{g-4}\rho_{g-7}^{-1}\cdots \rho_{7}^{-1}\rho_{4}\rho_{1}^{-1} & \text{($m$ is even)}. 
\end{cases}
\end{equation*}

$\phi$ acts on the curves on $\Sigma_{g}$ as follow:
\begin{center}
\begin{tabular}{lcr}
$\alpha_{2}=\phi(x_{1}),$\\
$\phi^{3}(\alpha^{\prime}_{1})=\phi^{2}(\gamma_{1})=\phi(\gamma_{2})=\alpha_{3},$& $\phi^{2}(\beta_{1})=\phi(\beta_{2})=\beta_{3},$\\
$\alpha^{\prime}_{4}=\phi(\gamma_{4})=\phi^{2}(\gamma_{5})=\phi^{3}(\alpha_{6}),$& $\beta_{4}=\phi(\beta_{5})=\phi^{2}(\beta_{6}),$\\
$\ \ \ \ \ \ \vdots$&$\vdots$\\
$\phi^{3}(\alpha^{\prime}_{g-4})=\phi^{2}(\gamma_{g-4})=\phi(\gamma_{g-3})=\alpha_{g-2},$& $\phi^{2}(\beta_{g-4})=\phi(\beta_{g-3})=\beta_{g-2}$& $(m \ is \ odd),$\\
\\
$(\alpha^{\prime}_{g-4}=\phi(\gamma_{g-4})=\phi^{2}(\gamma_{g-3})=\phi^{3}(\alpha_{g-2}),$& $\beta_{g-4}=\phi(\beta_{g-3})=\phi^{2}(\beta_{g-2})$& $(m \ is \ even)).$\\
\end{tabular}
\end{center}

\

We define
\begin{equation*}
\psi=
\begin{cases}
\rho_{g-2}^{-1}\rho_{g-5}\cdots \rho_{9}^{-1}\rho_{6}\sigma_{5}^{-1}\sigma_{4}\rho^{\prime -1} & \text{($m (=\frac{g-2}{3})$ is odd)} \nonumber\\
\rho_{g-2}\rho_{g-5}^{-1}\cdots \rho_{9}^{-1}\rho_{6}\sigma_{5}^{-1}\sigma_{4}\rho^{\prime -1} & \text{($m$ is even)}. \nonumber
\end{cases}
\end{equation*}

We note that $\psi$ acts on the curves on $\Sigma_{g}$ as follow:
\begin{center}
\begin{tabular}{lcr}
$\alpha_{2}=\psi(x_{2}),$& $\gamma_{2}=\psi(\alpha_{3}),$\\
$\psi^{3}(\alpha^{\prime}_{g-2})=\psi^{2}(\gamma_{g-2})=\psi(\gamma_{g-1})=\alpha_{g},$& $\psi^{2}(\beta_{g-2})=\psi(\beta_{g-1})=\beta_{g}$& $(m \ is \ odd),$\\
\\
$(\alpha^{\prime}_{g-2}=\psi(\gamma_{g-2})=\psi^{2}(\gamma_{g-1})=\psi^{3}(\alpha_{g}),$& $\beta_{g-2}=\psi(\beta_{g-1})=\psi^{2}(\beta_{g})$& $(m \ is \ even)).$
\end{tabular}
\end{center}

\

We define
\begin{equation*}
\omega=
\begin{cases}
\sigma_{g}\tau_{g-3}^{-1}\tau_{g-6}\cdots \tau_{8}\tau_{5}^{-1}\tau_{2}\sigma_{1}^{-1} & \text{($m (=\frac{g-2}{3})$ is odd)} \\
\sigma_{g}^{-1}\tau_{g-3}\tau_{g-6}^{-1}\cdots \tau_{8}\tau_{5}^{-1}\tau_{2}\sigma_{1}^{-1} & \text{($m$ is even)}. 
\end{cases}
\end{equation*}

$\omega$ acts on the curves on $\Sigma_{g}$ as follow:
\begin{center}
\begin{tabular}{lcr}
$\alpha_{1}=\omega(\beta_{1}),$\\
$\omega^{3}(\alpha_{2})=\omega^{2}(\gamma_{2})=\omega(\gamma_{3})=\alpha^{\prime}_{4},$& $\omega^{2}(\beta_{2})=\omega(\beta_{3})=\beta_{4},$\\
$\alpha^{\prime}_{5}=\omega(\gamma_{5})=\omega^{2}(\gamma_{6})=\omega^{3}(\alpha^{\prime}_{7}),$& $\beta_{5}=\omega(\beta_{6})=\omega^{2}(\beta_{7}),$\\
$\ \ \ \ \ \ \vdots$&$\vdots$\\
$\omega^{3}(\alpha_{g-3})=\omega^{2}(\gamma_{g-3})=\omega(\gamma_{g-2})=\alpha^{\prime}_{g-1},$& $\omega^{2}(\beta_{g-3})=\omega(\beta_{g-2})=\beta_{g-1}$& $(m \ is \ odd),$\\
\\
$(\alpha_{g-3}=\omega(\gamma_{g-3})=\omega^{2}(\gamma_{g-2})=\omega^{3}(\alpha^{\prime}_{g-1}),$& $\beta_{g-3}=\omega(\beta_{g-2})=\omega^{2}(\beta_{g-1})$& $(m \ is \ even)).$
\end{tabular}
\end{center}

\subsubsection{The genus is 5}
We assume that $g=5$.

We define 
\begin{eqnarray*}
\phi=(T_{\alpha_{5}}T_{\beta_{5}}T_{\gamma_{4}})(T_{\gamma_{3}}T_{\beta_{3}}T_{\gamma_{2}}T_{\beta{2}}T_{\gamma_{1}}T_{\beta_{1}}T_{\alpha^{\prime}_{1}})^{-2}.
\end{eqnarray*}
By \textit{chain reltion}, $(T_{\alpha_{5}}T_{\beta_{5}}T_{\gamma_{4}})^{4}$$=$$(T_{\gamma_{3}}T_{\beta_{3}}T_{\gamma_{2}}T_{\beta{2}}T_{\gamma_{1}}T_{\beta_{1}}T_{\alpha^{\prime}_{1}})^{8}$$=$$T_{\alpha_{4}}T_{\alpha^{\prime}_{4}}$
Therefore, we can find that $\phi^{4}=1$.

$\phi$ acts on the curves on $\Sigma_{g}$ as follow :
\begin{eqnarray*}
\phi^{3}(\alpha^{\prime}_{1})=\phi^{2}(\gamma_{1})=\phi(\gamma_{2})=\alpha_{3}, &&\phi^{2}(\beta_{1})=\phi(\beta_{2})=\beta_{3}, \\
\alpha_{2}=\phi(x_{1}), &&\phi(\gamma_{4})=\beta_{5}.
\end{eqnarray*}

\

We define
\begin{eqnarray*}
\psi=\sigma_{5}\sigma_{4}\rho^{\prime -1}.
\end{eqnarray*}
We can find that $\psi^{4}=1$.

$\psi$ acts on the curves on $\Sigma_{g}$ as follow :
\begin{eqnarray*}
\psi^{3}(\alpha^{\prime}_{1})=\psi^{2}(\gamma_{1})=\psi(\alpha_{3})=\gamma_{2}, &&\psi(\alpha_{2})=x_{2}.
\end{eqnarray*}

\

Let $s_{1}$, $s_{2}$ the non separating simple closed curves as shown in Figure~\ref{fig30}
\begin{figure}[htbp]
 \begin{center}
  \includegraphics*[width=5cm]{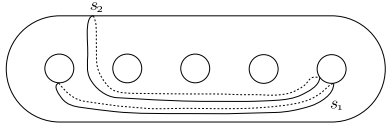}
 \end{center}
 \caption{the curves $s_{1}$ and $s_{2}$.}
 \label{fig30}
\end{figure}

We define
\begin{eqnarray*}
\omega=(T_{\alpha_{1}}T_{\beta_{1}}T_{s_{1}})(T_{\gamma_{4}}T_{\beta_{4}}T_{\gamma_{3}}T_{\beta_{3}}T_{\gamma_{2}}T_{\beta_{2}}T_{\alpha_{2}})^{-2}.
\end{eqnarray*}
By the \textit{chain relation}, $(T_{\alpha_{1}}T_{\beta_{1}}T_{s_{1}})^{4}$$=$$(T_{\gamma_{4}}T_{\beta_{4}}T_{\gamma_{3}}T_{\beta_{3}}T_{\gamma_{2}}T_{\beta_{2}}T_{\alpha_{2}})^{8}$$=$$T_{\alpha_{g}}T_{s_{2}}$.
Therefore, we can find that $\omega^{4}=1$.

$\omega$ acts on the curves on $\Sigma_{g}$ as follow :
\begin{eqnarray*}
\omega^{3}(\alpha_{2})=\omega^{2}(\gamma_{2})=\omega(\gamma_{3})=\gamma_{4}, &&\omega^{2}(\beta_{2})=\omega(\beta_{3})=\beta_{4}.
\end{eqnarray*}

\subsection{Generating the Dehn twist by 4 elements of order 4}
By using the lantern relation we generate the Dehn twist by 4 elements of order 4.

Since $\omega$ maps $\alpha_{2}$ to $\gamma_{2}$, we see that
\begin{eqnarray*}
T_{\alpha_{2}}=\omega T_{\gamma_{2}}\omega^{-1}
\end{eqnarray*}
and
\begin{eqnarray*}
T_{\alpha_{2}}T_{\gamma_{2}}^{-1}=\omega T_{\gamma_{2}}\omega^{-1}T_{\gamma_{2}}^{-1}=\omega (T_{\gamma_{2}}\omega^{-1}T_{\gamma_{2}}^{-1}).
\end{eqnarray*}
Let $\tilde{\omega}$ denote $T_{\gamma_{2}}\omega^{-1}T_{\gamma_{2}}^{-1}$.
Then we see that $T_{\alpha_{2}}T_{\gamma_{2}}^{-1}=\omega \tilde{\omega}$.

In the case of $g\neq 5$, we rewrite the lantern relation as follow :
\begin{eqnarray}\label{eq6}
T_{\alpha_{1}}=T_{\alpha_{1}^{\prime}}=(T_{\alpha_{2}}T_{\gamma_{2}}^{-1})(T_{x_{1}}T_{\gamma_{1}}^{-1})(T_{x_{2}}T_{\alpha_{3}}^{-1}).
\end{eqnarray}
From the argument of Section 4.1.1, 4.1.2 and 4.1.3 we can find that $\phi^{-1}(\alpha_{2})=x_{1}$, $\phi^{-1}(\gamma_{2})=\gamma_{1}$, $\psi^{-1}(\alpha_{2})=x_{2}$ and $\psi^{-1}(\gamma_{2})=\alpha_{3}$.
By Lemma~\ref{lem1} show we see that :
\begin{eqnarray*}
(T_{x_{1}}T_{\gamma_{2}}^{-1})=\phi^{-1}(T_{\alpha_{2}}T_{\gamma_{1}}^{-1})\phi \\ 
(T_{x_{2}}T_{\alpha_{3}}^{-1})=\psi^{-1}(T_{\alpha_{2}}T_{\gamma_{1}}^{-1})\psi. 
\end{eqnarray*}
We can rewrite (\ref{eq6}) as
\begin{eqnarray*}
T_{\alpha_{1}}=(\omega \tilde{\omega})(\phi^{-1} \omega \tilde{\omega}\phi)(\psi^{-1} \omega \tilde{\omega}\psi).
\end{eqnarray*}

In the case of $g=5$, we rewrite the lantern relation as follow :
\begin{eqnarray}\label{eq9}
T_{\alpha_{3}}=(T_{\alpha_{2}}T_{\gamma_{2}}^{-1})(T_{x_{1}}T_{\gamma_{1}}^{-1})(T_{x_{2}}T_{\alpha_{1}}^{-1}).
\end{eqnarray}
From the argument of Section 4.1.3 we can find that $\phi^{-1}(\alpha_{2})=x_{1}$, $\phi^{-1}(\gamma_{2})=\gamma_{1}$, $\psi(\alpha_{2})=x_{2}$ and $\psi(\gamma_{2})=\alpha_{1}^{\prime}=\alpha_{1}$.
By Lemma~\ref{lem1} we see that :
\begin{eqnarray*}
(T_{x_{1}}T_{\gamma_{1}}^{-1})=\phi^{-1}(T_{\alpha_{2}}T_{\gamma_{2}}^{-1})\phi \\ 
(T_{x_{2}}T_{\alpha_{1}}^{-1})=\psi(T_{\alpha_{2}}T_{\gamma_{2}}^{-1})\psi^{-1}.
\end{eqnarray*}
We can rewrite (\ref{eq9})
\begin{eqnarray*}
T_{\alpha_{3}}=(\omega \tilde{\omega})(\phi^{-1} \omega \tilde{\omega}\phi)(\psi \omega \tilde{\omega}\psi^{-1}).
\end{eqnarray*}

Hence Dehn tiwst is a product of 4 elements of order 4.

\subsection{Proof that 4 elements of order 4 generate}
We prove Theorem.
\begin{theorem}
If $g$ is at least 3, ${\cal M}_{g}$ can be generated by $\phi$, $\psi$, $\omega$ and $\tilde{\omega}$.
\end{theorem}
\begin{proof}
Let $G_{2}$ denote the group generated by $\phi$, $\psi$, $\omega$ and $\tilde{\omega}$.

Let $\alpha$ and $\beta$ be simple closed curves on $\Sigma_{g}$.
The symbol $\alpha \overset{\phi}{\longleftrightarrow} \beta$ (resp. $\alpha \overset{\psi}{\longleftrightarrow} \beta$, $\alpha \overset{\omega}{\longleftrightarrow} \beta$) means that either $\phi(\alpha)=\beta$ or $\phi^{-1}(\alpha)=\beta$ (resp. either $\psi(\alpha)=\beta$ or $\psi^{-1}(\alpha)=\beta$, either $\omega(\alpha)=\beta$ or $\omega^{-1}(\alpha)=\beta$).

If $g\neq 5$, we can find that $T_{\alpha_{1}}$ is in $G_{2}$.
In the case of $g=3m$ and $g=3m+1$, $\phi$ and $\omega$ can map $\alpha_{1}$ to all $\beta_{i}$ and $\gamma_{i}$ as shown in Figure~\ref{fig31}.
\begin{figure}[htbp]
 \begin{center}
  \includegraphics*[width=12cm]{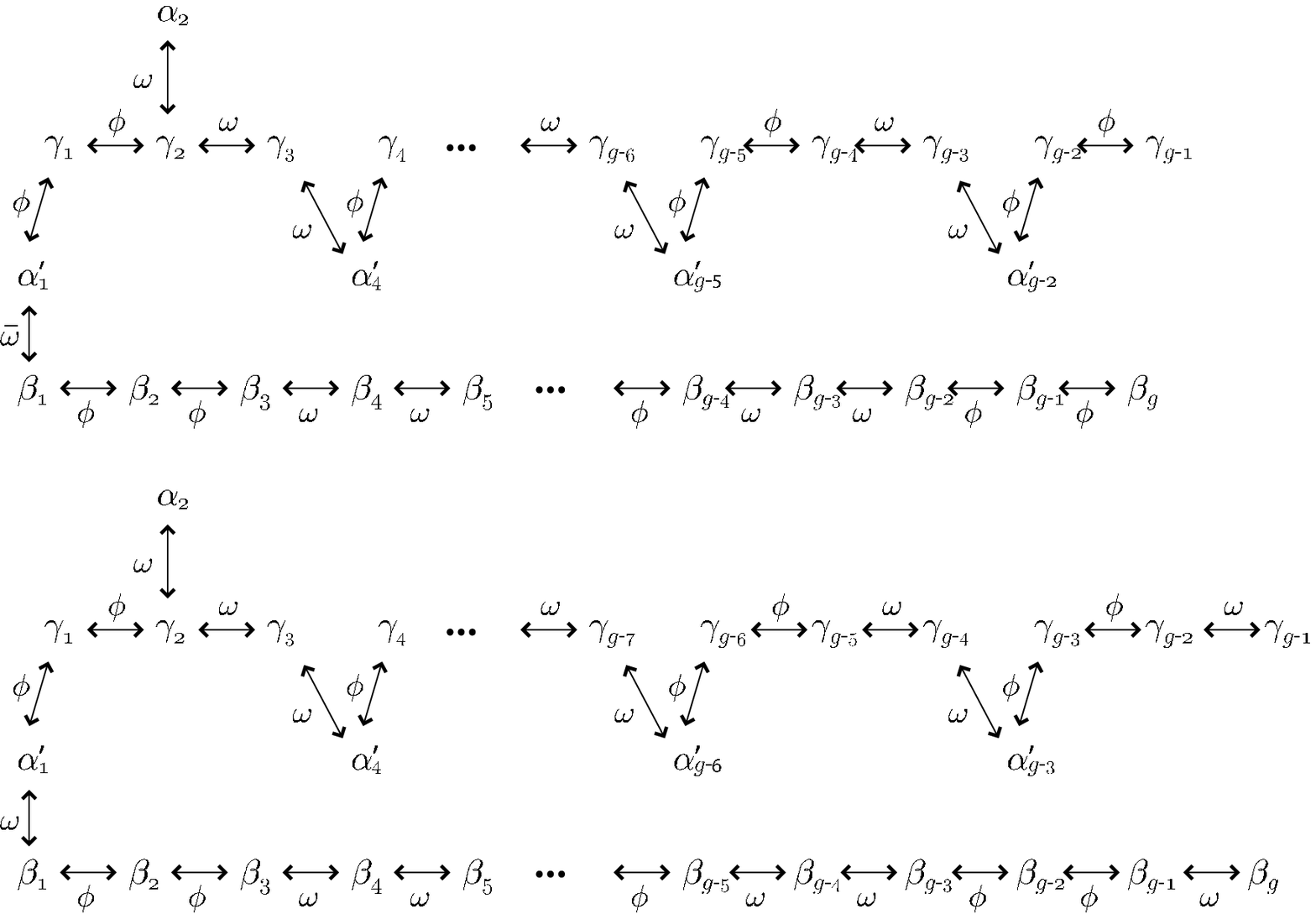}
 \end{center}
 \caption{}
 \label{fig31}
\end{figure}
If $g=3m+2$, we need $\psi$ other than $\phi$, $\omega$ like Figure~\ref{fig32}.
\begin{figure}[htbp]
 \begin{center}
  \includegraphics*[width=12cm]{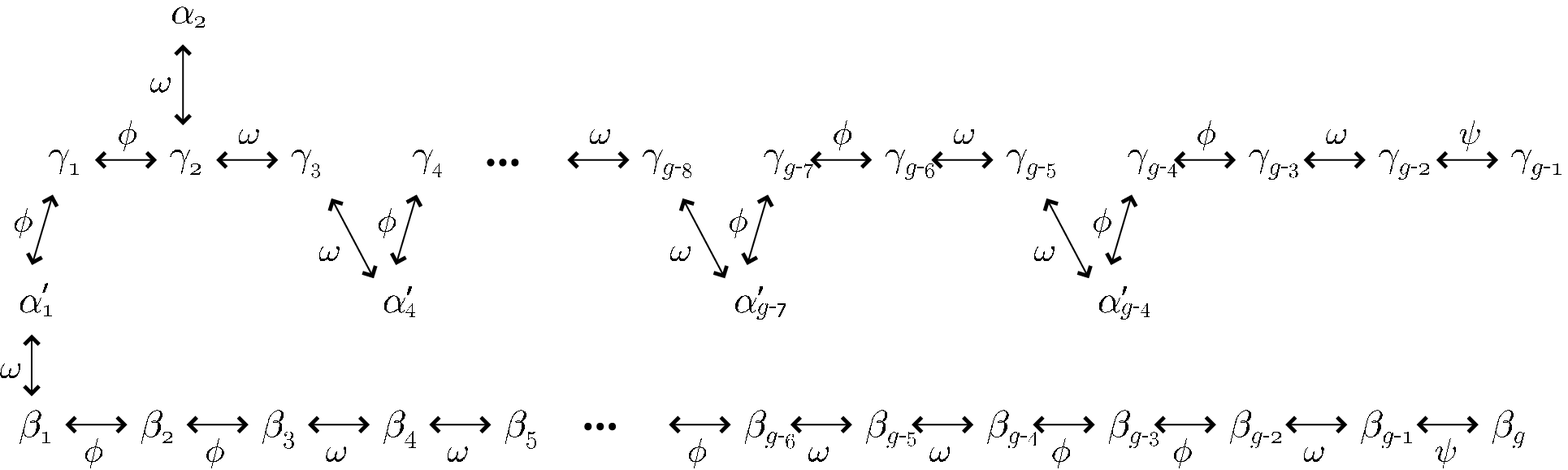}
 \end{center}
 \caption{}
 \label{fig32}
\end{figure}
Therefore, for all $i$, $T_{\beta_{i}}$ $T_{\gamma_{i}}$, are in $G_{2}$.
Because $\omega(\alpha_{2})=\gamma_{2}$, we can find that $T_{\alpha_{2}}\in G_{2}$. 
Therefore, all Humphries's generators are in $G_{2}$.

If $g=5$, $T_{\alpha_{3}}$ is in $G_{2}$.
The Figure~\ref{fig33} shows that we can send $\alpha_{3}$ to all Humphries's generators by $\phi$, $\psi$ and $\omega$.
\begin{figure}[htbp]
 \begin{center}
  \includegraphics*[width=4.5cm]{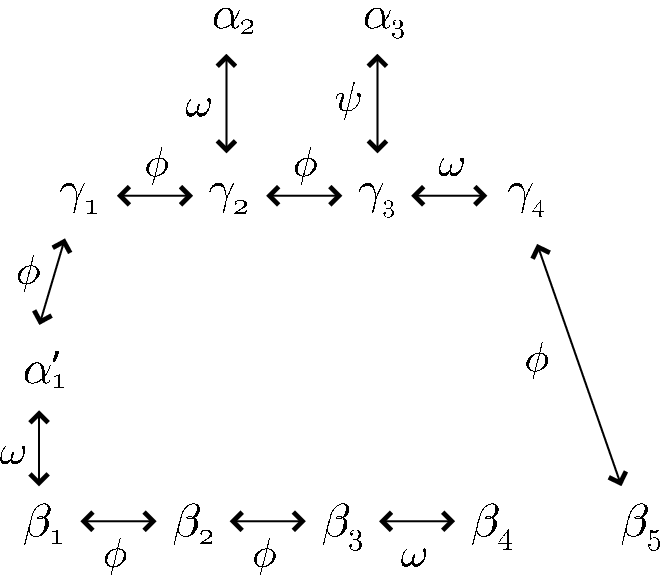}
 \end{center}
 \caption{}
 \label{fig33}
\end{figure}

We prove that $G_{2}$ is equal to ${\cal M}_{g}$ for $g\geq 3$.
\end{proof}

\section{Remarks}
\subsection{Low genus}
In the case of $g=1, 2$, we note that ${\cal M}_{g}$ can not be generated by elements of same order.
By using the argument of MacCarthy and Papadopulos [MP] and the work of Hirose [Hi], we can see the proof.
Hirose listed the Dehn twist presentation of finite order elements for closed oriented surfaces of genera up to 4.
We introduce the presentation of finite order elements in the case of $g=1, 2$.
The list is as follow :
\\
\
\\
\begin{tabular}{c||d{40.0}|d{4.0}}
\hline
 \textbf{genus} & \multicolumn{1}{c|}{\textbf{elements}} & \multicolumn{1}{c}{\textbf{order}} \\
\hline
 1                                     & T_{\beta_{1}}T_{\alpha_{1}}& 6\\
                                       & T_{\alpha_{1}}T_{\beta_{1}}T_{\alpha_{1}} &4\\
\hline
 2                                     & T_{\beta_{2}}T_{\gamma_{1}}T_{\beta_{1}}T_{\alpha_{1}}&  10\\
                                       & T_{\beta_{2}}T_{\beta_{2}}T_{\gamma_{1}}T_{\beta_{1}}T_{\alpha_{1}}&  8\\
                                       & T_{\alpha_{2}}T_{\beta_{2}}T_{\gamma_{1}}T_{\beta_{1}}T_{\alpha_{1}}&  6\\
                                       & (T_{\alpha_{1}}T_{\beta_{1}}T_{\gamma_{1}}T_{\beta_{2}}T_{\alpha_{2}})(T_{\alpha_{2}}T_{\beta_{2}}T_{\gamma_{1}}T_{\beta_{1}}T_{\alpha_{1}})^{3}&  6\\
\hline
\end{tabular}

\

MacCarthy and Papadopulos proved that ${\cal M}_{2}$ can not be generated by elements of order 2.
The argument of MacCarthy and Papadopulos is as follow:
\begin{proof}
Let $c$ be a nonseparating simple closed curve and $p$ be the abelianization map given by Powell's result [Po]:

$$
\begin{array}{ccc}
\textit{p} : {\cal M}_{2}                     & \longrightarrow & \mathbb{Z}_{10} \\[-4pt]
\\
\ \ \ \ \rotatebox{90}{$\in$} &                 & \rotatebox{90}{$\in$} \\[-4pt]
\\
\ \ \ \ \textit{T}_{\textit{c}}                     & \longmapsto     & 1
\end{array}
$$
We can find that
\begin{eqnarray*}
p((T_{\beta_{2}}T_{\gamma_{1}}T_{\beta_{1}}T_{\alpha_{1}})^5)&=&p((T_{\beta_{2}}T_{\beta_{2}}T_{\gamma_{1}}T_{\beta_{1}}T_{\alpha_{1}})^{4})\\
&=&p(\{(T_{\alpha_{1}}T_{\beta_{1}}T_{\gamma_{1}}T_{\beta_{2}}T_{\alpha_{2}})(T_{\alpha_{2}}T_{\beta_{2}}T_{\gamma_{1}}T_{\beta_{1}}T_{\alpha_{1}})^{3}\}^{3})\\
&=&0\\
p((T_{\alpha_{2}}T_{\beta_{2}}T_{\gamma_{1}}T_{\beta_{1}}T_{\alpha_{1}})^{3})&=&5.
\end{eqnarray*}
It is easily to see that $\mathbb{Z}_{10}$ can not be generated by 0 and 5.
Therefore we can see that ${\cal M}_{2}$ can not be generated by elements of order 2.
\end{proof}

By the similar proof, we can see that ${\cal M}_{1}$ and ${\cal M}_{2}$ can not be generated by elements of same order.
\begin{remark}
${\cal M}_{1}$ and ${\cal M}_{2}$ can be generated by elements of different order.
For example, ${\cal M}_{1}$ can be generated by $ T_{\beta_{1}}T_{\alpha_{1}}$ and $T_{\alpha_{1}}T_{\beta_{1}}T_{\alpha_{1}}$, and ${\cal M}_{2}$ can be generated by $T_{\beta_{2}}T_{\gamma_{1}}T_{\beta_{1}}T_{\alpha_{1}}$ and $T_{\alpha_{2}}T_{\beta_{2}}T_{\gamma_{1}}T_{\beta_{1}}T_{\alpha_{1}}$.
\end{remark}

\subsection{Lower bound}
The order of ${\cal M}_{g}$ is not finite.
Therefore, we can easily find that lower bound of the number of generators whose order are 3 (resp. 4) is 2.
The author has the following question :
\begin{question}
What is the minimal number of elements of order 3 (resp. 4) required to generate ${\cal M}_{g}$?
\end{question}

Department of Mathematics, Graduate School of Science, Osaka University, Toyonaka, Osaka 560-0043, Japan\\
\textit{E-mail adress:} \bf{n-monden@cr.math.sci.osaka-u.ac.jp}
\end{document}